\documentclass[a4,12pt]{amsart}
\usepackage[left=0.8 in, right=0.8 in, top=1 in, bottom=1 in]{geometry}

\usepackage{amssymb,amstext,amsmath,amscd,amsthm,amsfonts,enumerate,latexsym}
\usepackage{color}
\usepackage[dvipdfmx]{graphicx}
\usepackage[all]{xy}
\usepackage{stmaryrd} 
\usepackage{ytableau}
\usepackage{hyperref}

\usepackage{bm}

\theoremstyle{plain}
\newtheorem{thm}{Theorem}[section]

\newtheorem*{thm*}{Theorem}
\newtheorem*{cor*}{Corollary}

\newtheorem{lem}[thm]{Lemma}
\newtheorem{cor}[thm]{Corollary}

\newtheorem*{claim*}{Claim}

\theoremstyle{definition}

\newtheorem{ex}[thm]{Example}

\newtheorem{remark}[thm]{Remark}

\theoremstyle{remark}

\numberwithin{equation}{thm}

\def\Z{\mathbb{Z}}

\def\rank{\operatorname{rank}}

\def\m{\mathfrak m}

\def\p{\mathfrak p}

\newcommand{\calA}{\mathcal{A}}

\newcommand{\calR}{\mathcal{R}}
\newcommand{\calS}{\mathcal{S}}

\newcommand{\fkm}{\mathfrak{m}}

\newcommand{\fkp}{\mathfrak{p}}

\newcommand{\mapright}[1]{%
\smash{\mathop{%
\hbox to 1cm{\rightarrowfill}}\limits^{#1}}}

\newcommand{\mapleft}[1]{%
\smash{\mathop{%
\hbox to 1cm{\leftarrowfill}}\limits_{#1}}}

\def\depth{\operatorname{depth}}

\def\height{\mathrm{ht}}
\def\grade{\mathrm{grade}}
\def\Spec{\operatorname{Spec}}

\def\gr{\mbox{\rm gr}}

\def\ds{\displaystyle}
\def\tratto{\mbox{\rule{3mm}{.2mm}$\;\!$}}

\title[Normality of Ideals and Modules]{Normality of Ideals and Modules}

\author[Naoki Endo]{Naoki Endo}
\address{School of Political Science and Economics, Meiji University, 1-9-1 Eifuku, Suginami-ku, Tokyo 168-8555, Japan}
\email{endo@meiji.ac.jp}
\urladdr{https://www.isc.meiji.ac.jp/~endo/}

\author[Shiro Goto]{Shiro Goto}
\address{Department of Mathematics, School of Science and Technology, Meiji University, 1-1-1 Higashi-mita, Tama-ku, Kawasaki 214-8571, Japan}
\email{shirogoto@gmail.com}

\author[Jooyoun Hong]{Jooyoun Hong}
\address{Department of Mathematics, Southern Connecticut State University, 501 Crescent Street, New Haven, CT 06515-1533, USA}
\email{hongj2@southernct.edu}

\author[Bernd Ulrich]{Bernd Ulrich}
\address{Department of Mathematics, Purdue University, West Lafayette, IN 47907 USA}
\email{bulrich@purdue.edu}

\thanks{The first author was partially supported by JSPS Grant-in-Aid for Scientific Research (C) 23K03058.} 
\thanks{The fourth author was partially supported by NSF grant DMS-2502706.}

\begin{document}

\maketitle

\setlength{\baselineskip}{15.3pt}

\begin{abstract}
We investigate when the Rees algebra of an integrally closed $\m$-primary ideal in a regular local ring is a Cohen–Macaulay normal domain. While this property always holds in dimension two, it fails in general in higher dimensions, prompting a search for sufficient conditions on the ideal. We show that if an integrally closed ideal contains a part of  regular system of parameters of length $d-2$, where $d$ is the dimension of the regular local ring, then its Rees algebra is Cohen–Macaulay and normal. We also extend results of Goto and Ciuperc\u{a} by proving the same conclusion when the minimal number of generators of an ideal is at most $d+2$.
Furthermore, we treat the case of integrally closed zero-dimensional ideals generated by $d+3$ homogeneous polynomials. Finally, using generic Bourbaki ideals, we generalize these results to integrally closed torsionfree modules of finite colength.
\end{abstract}

\section{Introduction}\label{sec1}

In this paper, we investigate the conditions under which the Rees algebra of an integrally closed $\m$-primary ideal $I$ in a regular local ring $(R, \m)$ of dimension $d$ is a Cohen-Macaulay normal domain. Since the statement holds trivially when $d \le 1$, our primary focus is on the case $d \ge 2$. Recall that the {\em integral closure} $\overline{I}$ of an ideal $I$ in a commutative ring $R$ is the set of all elements $r$ that are integral over $I$, that is, satisfy a polynomial equation of the form
\[ r^{m} + a_{1} r^{m-1} + \cdots+ a_{j} r^{m-j} + \cdots + a_{m} =0, \]
where ${\ds a_{j} \in I^{j}}$. An ideal $I$ is said to be {\em integrally closed} if $I=\overline{I}$. Furthermore, $I$ is called {\em normal} if ${\ds I^{n} = \overline{I^{n}}}$ for every positive integer $n$. Alternatively, we consider the {\em Rees algebra} $\calR(I)$ of $I$, which is the subalgebra of the polynomial ring $R[t]$ defined by
\[ \calR(I) =  \bigoplus_{n \geq 0} I^{n} t^{n}. \] The integral closure ${\ds \overline{\calR(I)}}$ of $\calR(I)$ in $R[t]$ is also a graded algebra, and its graded components are the integral closures of all powers of $I$:
\[ \overline{\calR(I)} =  \bigoplus_{n \geq 0} \overline{I^{n}} t^{n}. \]
Provided that $R$ is a normal domain, the integral closure of $\calR(I)$ in its field of fractions coincides with ${\ds \overline{\calR(I)}}$. Accordingly, the Rees algebra $\calR(I)$ is said to be {\em normal} if ${\ds \calR(I) = \overline{\calR(I)}}$.
Therefore, an ideal $I$ is normal if and only if its Rees algebra is normal.

\medskip

Zariski demonstrated that in the polynomial ring with two variables over a field, the product of any two integrally closed ideals is also integrally closed. This implies that the Rees algebra is normal (\cite[Part II, Section 12]{Z}). Although Zariski's original result was stated for the polynomial ring, the analogous property for two-dimensional regular local rings was established in \cite[Appendix 5, Theorem 2$'$]{ZS}. Furthermore, Lipman and Teissier proved that the reduction number of any integrally closed $\m$-primary ideal is at most one (\cite[Proposition 5.5]{LT}). Consequently, the Rees algebra is Cohen-Macaulay. Therefore, if $(R, \m)$ is a two-dimensional regular local ring and $I$ is an integrally closed ideal, then the Rees algebra $\calR(I)$ is indeed a Cohen-Macaulay normal domain. In addition, Lipman showed that if $R$ is a two-dimensional rational singularity with an infinite residue field, then the Rees algebra $\calR(I)$ is a Cohen-Macaulay normal domain (\cite[Section 7]{L69}). 

\medskip

However, if the dimension of the ring is greater than $2$, then the Rees algebra of an integrally closed ideal is not necessarily a Cohen-Macaulay normal domain. For instance, consider the formal power series ring $R = k[\![X, Y, Z]\!]$ over a field $k$, and examine the ideals 
\[
Q = (X^7, Y^3, Z^2) \ \ \text{and} \ \ I = \overline{Q} = (X^7, Y^3, Z^2, X^5Y, X^4Z, X^3Y^2, X^2YZ, Y^2Z).
\]
In this case, $I$ is integrally closed, but $\overline{I^2} \ne I^2$, and $I^2 = QI$. Thus, $\calR(I)$ is Cohen-Macaulay but not normal. While one might expect that normality implies the Cohen-Macaulay property, this implication does not hold in general. Suppose further that the characteristic of $k$ is not equal to $3$, and consider 
\[
I = (X^4, X(Y^3+Z^3), Y(Y^3+Z^3), Z(Y^3+Z^3)) + \m^5
\]
where $\m =(X, Y, Z)$. Then, by \cite[Theorem 3.11]{HH}, the ideal $I$ is normal, but the associated graded ring $\gr_{I}(R)=\bigoplus_{n \ge 0}I^n/I^{n+1}$ is not Cohen-Macaulay. Therefore, $\calR(I)$ is normal, but not Cohen-Macaulay. These examples suggest that additional conditions on either the base ring or the ideal are necessary for the Rees algebra to be a Cohen-Macaulay normal domain of dimension greater than $2$. As the base ring is assumed to be regular in our setup, we focus on imposing conditions on the ideals. 

\medskip

We investigate such conditions from the perspective of the minimal number of generators, denoted by $\mu_{R}(\tratto)$. Goto proved that when $\mu_R(I) = d$ -- that is, when $I$ is a complete intersection ideal -- the Rees algebra $\calR(I)$ is a Cohen-Macaulay normal domain (\cite[Corollary (1.3)]{G}). In the same work, it was also proved that an $\m$-primary complete intersection ideal in a regular local ring is integrally closed if and only if it contains a part of regular system of parameters of length $d - 1$ (\cite[Theorem (1.1)]{G}); equivalently, this is the case if and only if $v(R/I) \le 1$, where $v(\tratto)$ denotes the embedding dimension of a ring. Furthermore, Ciuperc\u{a} proved that $\calR(I)$ is a Cohen-Macaulay normal domain also when $\mu_R(I) = d + 1$, that is, when $I$ is an almost complete intersection ideal (\cite[Theorem 1.1]{C1}, \cite[Section 4]{C2}). In this case, it is also shown that if $I$ is integrally closed and $\mu_R(I) = d + 1$, then $I$ necessarily contains a part of  regular system of parameters of length $d - 2$ (\cite[Corollary 3.3]{C1}, \cite[Section 4]{C2}). The latter condition is equivalent to $v(R/I) \le 2$.

\medskip

Building upon these results, we broaden the scope by showing that the Rees algebra $\calR(I)$ remains a Cohen-Macaulay normal domain under the more general condition $v(R/I) \le 2$, even in cases where $\mu_R(I) > d + 1$ (See Theorem~\ref{main}). Furthermore, we extend the results of Goto and Ciuperc\u{a} by showing that if $\mu_{R}(I) \le d+2$, then its Rees algebra is a Cohen-Macaulay normal domain. Our main theorem is as follows:

\noindent {\bf Theorem~\ref{div2}} {\it
Let $(R, \m)$ be a regular local ring of dimension $d \ge 3$ and $I$ an integrally closed $\m$-primary ideal of $R$. If $\mu_R(I) \le d+2$, then  $I$ contains  a part of regular system of parameters of length $d-2$. In particular, the Rees algebra $\calR(I)$ is a Cohen-Macaulay normal domain.}

\medskip

For ideals generated by at most $d+3$ elements, we analyze a zero-dimensional ideal generated by homogeneous polynomials and obtain the following result. 

\medskip

\noindent {\bf Theorem~\ref{Ulrich}} {\it 
Let $k$ be a field of characteristic zero, $R=k[X_{1}, \ldots, X_{d}]$ a polynomial ring, and $I$ a zero-dimensional $R$-ideal generated by $d+3$ homogeneous polynomials. If $I$ is integrally closed, then the Rees algebra $\calR(I)$ is a Cohen-Macaulay normal domain.}

\medskip

Let $R$ be a Noetherian ring and $E$ a finitely generated $R$-module having a rank. The Rees algebra $\calR(E)$ of the module $E$ is defined as the symmetric algebra $\calS(E)$ modulo $R$-torsion. If $E$ is torsionfree, then it can be embedded into a free module $R^{e}$. In this case, $\calR(E)$ is the $R$-subalgebra of the polynomial ring  ${\ds R[t_{1}, \ldots, t_{e}] }$ generated by all linear forms $a_{1}t_{1} + \cdots + a_{e}t_{e}$, where $(a_{1}, \ldots, a_{e})$ is the image of an element of $E$ in $R^{e}$ under the embedding. In particular, when $E=I \subseteq R$ is an ideal, this definition of the Rees algebra coincides with the one given at the beginning of Introduction. The integral closure $\overline{E}$ of $E$ (in $R^{e}$) is the largest submodule of $R^{e}$ whose Rees algebra is integral over $\calR(E)$. If ${\ds E=\overline{E}}$, then $E$ is said to be {\em integrally closed}. Using the notion of a generic Bourbaki ideal, we extend our results from ideals to those of modules. For an $R$-module $M$, we denote the length of $M$ by $\ell_{R}(M)$, and the deviation of $M$ by $d(M)$ (See \cite[Page 633]{SUV}).

\medskip

\noindent {\bf Theorem~\ref{mod}} {\it Let $(R, \m)$ be a regular local ring of dimension $d \geq 3$, and  $E$  a finitely generated torsionfree $R$-module having a rank $e >0$.  Suppose that $E$ is integrally closed and  $\ell_R(R^{e}/E)< \infty$.
If $\mu_R([E + \m R^{e}]/E) \leq 2$ or $d(E) \leq 2$, then  the Rees algebra  $\calR(E)$ is a Cohen-Macaulay normal domain.
}

\section{Cohen-Macaulay normal Rees algebras}

A local ring $(R, \m)$ of dimension $d$ is called {\em regular} if the maximal ideal $\m$ can be generated by exactly $d$ elements. In this case, a minimal system of generators for $\m$ is called a {\em regular system of parameters}. Note that an $R$-ideal $I$ contains a part of regular system of parameters of length $d-2$ if and only if ${\ds \ell_R( (I+\m^2)/\m^2) \ge d-2}$.

\begin{thm}\label{main}
Let $(R, \m)$ be a regular local ring of dimension $d \geq 2$ and $I$ an integrally closed $\m$-primary ideal of $R$. Suppose that $I$ contains a part of regular system of parameters of length $d-2$. Then the Rees algebra $\calR(I)$ is a Cohen-Macaulay normal domain.
\end{thm}

\begin{proof}
By passing to the faithfully flat extension $R \to R[X]_{\m R[X]}$ if necessary, we may assume that the residue field $R/\m$ is infinite. We prove the assertion by induction on  $d$.

\medskip

\noindent Let $d=2$ and let $Q$ be a parameter ideal that is a minimal reduction of $I$. By \cite[Proposition 5.5]{LT} (see also \cite[Theorem 5.1]{H}), it follows that $I^2 = QI$. Then by \cite[Corollary 2.7]{VV} the associated graded ring $\gr_I(R)$ is Cohen-Macaulay. Hence, by \cite[Remark (3.10)]{GS}, the Rees algebra $\calR(I)$ is Cohen-Macaulay. Moreover, by \cite[Appendix 5, Theorem 2$'$]{ZS} (see also \cite[Theorem 3.7]{H}), the Rees algebra $\calR(I)$ is normal. 
 
\medskip  

\noindent Now assume that $d \ge 3$ and that the assertion holds true for $d-1$. Write $I = (a_1, a_2, \ldots, a_{d-2}, \ldots, a_n)$, where $a_{1}, \ldots, a_{d-2} $  forms a part of regular system of parameters of length $d-2$. Let ${\ds R'=R[z_{1}, \ldots, z_{n}]}$ be a polynomial ring with indeterminates $z_{i}$. Set
\[ R''= R'_{\m R'}, \quad I'' = IR'', \quad \m''=\m R'', \quad x= \sum_{i=1}^{n} z_{i} a_{i} \in I''.\] Since the natural map $\calR(I) \to R'' \otimes_R\calR(I) \cong \calR(I'')$ is faithfully flat, it is enough to show that $\calR(I'')$ is a Cohen-Macaulay normal domain. 

\medskip

\noindent Since ${\ds I \not\subseteq \m^2}$,  ${\ds x \not\in (\m'')^{2}}$. Therefore, 
$(R''/(x), \m''/(x))$ is a regular local ring of dimension $d-1$, and the ideal $I''/(x)$ is $\m''/(x)$-primary. By \cite[Corollary 2.3]{HU}, 
\[ \overline{(I''/(x))} = \overline{I''}/(x) = I''/(x). \]
That is, $I''/(x)$ is integrally closed. Also, we obtain the following.
\[ \begin{array}{rcl}
{\ds \ell_{R''} \left(  \frac{I''/(x) +(\m''/(x))^2}{(\m''/(x))^2} \right)  } &=& {\ds \ell_{R''}\left( \frac{I'' + (\m'')^2}{(\m'')^2 + (x)}  \right)  } \vspace{0.1 in} \\
&=& {\ds \ell_{R''} \left( (I'' +(\m'')^{2})/(\m'')^{2}  \right) - \ell_{R''} \left(  ((\m'')^{2}+(x)  )/(\m'')^{2} \right)  } \vspace{0.1 in} \\
&=& {\ds \ell_{R''} \left( R''/(\m'')^{2}  \right) - \ell_{R''} \left( R''/(I'' +(\m'')^{2}) \right) -  \ell_{R''} \left(  ((\m'')^{2}+(x)  )/(\m'')^{2} \right) } \vspace{0.1 in} \\
&=& {\ds \ell_{R} \left(R/\m^{2}\right) - \ell_{R}\left( R/(I+\m^{2})  \right) -1} \vspace{0.1 in} \\
&=& {\ds \ell_{R} \left( (I+\m^{2}) / \m^{2} \right) - 1 } \vspace{0.1 in} \\
&\geq & {\ds d-3}
\end{array} \]
By the induction hypothesis, the Rees algebra ${\ds \calR(I''/(x))}$ is a Cohen-Macaulay normal domain. Then, by \cite[Remark (3.10)]{GS} the associated graded ring $\gr_{I''/(x)}(R''/(x))$ is Cohen-Macaulay. 

\medskip

\noindent Consider the integral closure $\overline{\calA}$ of the extended Rees algebra of $I''$ in $R''[t, t^{-1}]$: 
\[ \overline{\calA} = \sum_{n \in \Bbb Z}\overline{(I'')^n}t^n. \]
By \cite[Lemma 1.1 (c)]{HU}, the element $xt$ is regular on $ \overline{\calA}/t^{-1} \overline{\calA}$. Then the equality 
\[ \overline{(I'')^n}:_{R''}x = \overline{(I'')^{n-1}} \]
holds for all $n \ge 1$. Also, since $\calR(I''/(x))$ is normal, we have 
\[ \overline{(I''/(x))^n} = (I''/(x))^n. \]
Thus, $\overline{(I'')^n} \subseteq (I'')^n + (x)$ for all $n \ge 1$. 
Now, we prove that ${\ds \overline{(I'')^n} = (I'')^{n}}$ for all $n \geq 1$ by induction. It is trivial if $n=1$. Let $n \geq 2$ and assume that assertion holds true for $n-1$. Then
\[ \begin{array}{rcl}
{\ds \overline{(I'')^n} } &=& {\ds \big(  (I'')^n + (x) \big) \cap  \overline{(I'')^n} } \vspace{0.1 in} \\
&=& {\ds  (I'')^n + \big( (x) \cap \overline{(I'')^n} \big)   }  \vspace{0.1 in} \\
&=& {\ds  (I'')^{n} + x \big( \overline{(I'')^n}:_{R''}x  \big) } \vspace{0.1 in} \\
&=&  {\ds (I'')^{n} + x \overline{(I'')^{n-1}} } \vspace{0.1 in} \\
&=& {\ds (I'')^{n} + x (I'')^{n-1} } \vspace{0.1 in} \\
&=& {\ds (I'')^{n} }
\end{array}\]
Therefore, Rees algebra $\calR(I'')$ is normal.

\medskip

\noindent Since $xt$ is regular on ${\ds \overline{\calA}/t^{-1} \overline{\calA} = \gr_{I''}(R'')}$, it follows that 
\[
\gr_{I''}(R'')/(xt)\gr_{I''}(R'') \cong \gr_{I''/(x)}(R''/(x)).
\] Since $\gr_{I''/(x)}(R''/(x))$ is Cohen-Macaulay, $\gr_{I''}(R'')$ is Cohen-Macaulay. 
Therefore, by \cite[Theorem (5)]{L}, the Rees algebra $\calR(I'')$ is Cohen-Macaulay.
\end{proof}

The following theorem extends the theorem of Goto on complete intersection ideals (\cite[Corollary (1.3)]{G}), as well as the result of Ciuperc\u{a} on almost complete intersection ideals (\cite[Theorem 1.1]{C1}, \cite[Section 4]{C2}).

\begin{thm}\label{div2}
Let $(R, \m)$ be a regular local ring of dimension $d \ge 3$ and $I$ an integrally closed $\m$-primary ideal of $R$. If $\mu_R(I) \le d+2$, then  $I$ contains  a part of regular system of parameters of length $d-2$. In particular, the Rees algebra $\calR(I)$ is a Cohen-Macaulay normal domain.
\end{thm}

\begin{proof}
Enlarging the residue class field $R/\m$ of $R$ if necessary, we may assume $R/\m$ is infinite. Suppose that $I \subseteq \m^2$. Because $I$ is $\m$-full (\cite[Theorem (2.4)]{G}), by \cite[Theorem 3]{W87} (see also \cite[Lemma (2.2)]{G}), we obtain the following.  
\[ d+2 \ge \mu_R(I) \ge \mu_R(\m^2) =  \frac{d(d+1)}{2}. \]
After simplifying the inequality and using $d \geq 3$, we get
\[ 0 \ge d^{2}-d-4 \geq 2. \]
This is a contradiction. Therefore, $I \not\subseteq \m^2$. 

\medskip

\noindent Now we prove the theorem by induction on $d$. Let $d=3$. Since $I \not\subseteq \m^2$, the ideal $I$ contains a part of regular system of parameters of length $1$.

\medskip

\noindent Let $d \geq 4$ and assume that  the assertion holds true for $d-1$. Choose $x \in I$ with $x \not\in \m^2$. By \cite[Lemma 4.2]{C2}, we obtain 
\[ \overline{(I/(x))} = \overline{I}/(x) = I/(x). \]
Now, $R/(x)$ is a regular local ring of dimension $d-1$ and the ideal $I/(x)$ is  integrally closed and $\m/(x)$-primary.   Moreover, 
\[ \mu_{R/(x)}(I/(x)) = \mu_R(I) - 1 \le (d+2) - 1 = (d-1) + 2. \]
By the induction hypothesis,  $I/(x)$ contains  a part of regular system of parameters of length $d-3$. Equivalently, ${\ds \ell_{R} \left( (I + \m^{2})/(\m^{2} +(x))  \right) \geq d-3}$. 
Therefore,
\[ \begin{array}{rcl}
 {\ds \ell_R \left( (I+\m^2) /\m^2 \right)}  &=&  {\ds \ell_R \left( (I+\m^2)/(\m^2+(x)) \right) + \ell_R \left( (\m^2+(x))/\m^2 \right) } \vspace{0.1 in} \\
& \ge & (d-3) + 1 = d-2,
\end{array} \]
which proves that $I$ contains a part of regular system of parameters of length $d-2$.
\end{proof}

\begin{ex}\label{3.5}
Let $R=k[\![X, Y, Z]\!]$ be the formal power series ring over a field $k$ and $\m = (X, Y, Z)$  the maximal ideal of $R$. The Rees algebras of the following ideals are Cohen-Macaulay normal domains. 
\begin{itemize}
\item[$(1)$] Let $I=\overline{(X^3, Y^3, Z)} = (X^3, X^2Y, XY^2, Y^3, Z)$. Then $I$ is an integrally closed $\m$-primary ideal of $R$ and $\mu_R(I) = 5$. 
\item[$(2)$] Let $I=\overline{(X^4, Y^4, Z)} = (X^4, X^3Y, X^2Y^2, XY^3, Y^4, Z)$. Then $I$ is an integrally closed $\m$-primary ideal of $R$ and $\mu_R(I) = 6\,>d+2$, but $v(R/I) = 2$. 
\item[$(3)$] Let $I=(f) + \m^n$ for each $f \in \m\setminus \m^2$ and $n \ge 1$. Then $I$ is an integrally closed $\m$-primary ideal of $R$ and $v(R/I) \le 2$. 
\end{itemize}
\end{ex}

\begin{cor}
Let $R$ be a regular ring of dimension $d$, and let $I$ be an integrally closed ideal of $R$ with $\height_{R}(I)=d$. If $\mu_{R}(I) \le d+2$, then the Rees algebra $\calR(I)$ is a Cohen-Macaulay normal domain. 
\end{cor}

\begin{proof}
Localizing at a maximal ideal $\m$ in $R$ containing $I$, we may assume that $(R, \m)$ is a regular local ring and that $I$ is $\m$-primary. Enlarging the residue class field $R/\m$ of $R$ if necessary, we may further assume that $R/\m$ is infinite. 
The assertion is immediate from the definition when $d \le 1$. In the case $d=2$, it follows from \cite[Appendix 5, Theorem 2$'$]{ZS} together with \cite[Proposition 5.5]{LT}. When $d\ge 3$, it is a consequence of Theorem \ref{div2}.
\end{proof}

\section{The case where $\mu_R(I) = d+3$ for monomial ideals $I$}

In this section, we investigate the normality of a particular class of ideals, namely homogeneous ideals in a polynomial ring. By leveraging the algebraic structure inherent in homogeneous ideals, we aim to characterize conditions under which they are normal. It is well known that  the integral closure of a monomial ideal is itself a monomial ideal. For a real number $r$, we denote ${\ds \lceil r \rceil= \min\left\{ n \in \Z \mid r \leq n  \right\}}$.

\begin{lem}\label{dim2mono}
Let ${\ds R=k[X, Z]}$ be a polynomial ring over a field $k$. Let ${\ds J=(X^{2}, XZ^{a}, Z^{c})}$ be an $R$-ideal, where $a \geq 1$ and $c \geq 2$ are integers such that ${\ds a \leq  \left\lceil \frac{c}{2} \right\rceil}$.  Then $J$ is normal.
\end{lem}

\begin{proof} It is enough to show that $J$ is integrally closed (\cite[Part II, Section 12]{Z}). Let ${\ds f = X^{\alpha}Z^{\beta} }$ be a monomial in the integral closure ${\ds \overline{J}}$, where $\alpha, \beta \geq 0$. We show that $f \in J$ by cases.

\medskip

\noindent Suppose that $\alpha =0$ or $\beta =0$. Then 
\[ f = Z^{\beta} \in \overline{\left(Z^{c}\right)} = \left(Z^{c}\right) \subseteq J, \quad \mbox{or} \quad f = X^{\alpha} \in \overline{\left(X^{2}\right)} = \left(X^{2}\right) \subseteq J. \]

\medskip

\noindent Suppose that $\alpha \geq 2$. Then
\[ f= X^{\alpha}Z^{\beta} = X^{2} \cdot X^{\alpha -2} Z^{\beta} \in J. \]

\medskip

\noindent Suppose that $\alpha =1$ and $\beta \geq a$. Then
\[ f= X Z^{\beta} = XZ^{a} \cdot Z^{\beta - a} \in J. \]

\medskip

\noindent Suppose that $\alpha =1$ and $ 1 \leq \beta \leq a-1$. Then there exists positive integer $\rho$ such that 
\[ \left( XZ^{\beta} \right)^{\rho} = h \cdot \left(X^{2}\right)^{n_{1}}\left(XZ^{a}\right)^{n_{2}}\left(Z^{c}\right)^{n_{3}}, \]
where $h$ is a monomial in $R$, and $n_{1}, n_{2}, n_{3}$ are nonnegative integers such that $n_{1}+n_{2}+n_{3} = \rho$.
Therefore,
\[ X^{\rho}Z^{\beta \rho}  = h \cdot X^{2n_{1}+n_{2}}  Z^{an_{2}+cn_{3}}.   \]
Then
\[ \rho \geq 2n_{1}+n_{2} \;\; \Rightarrow \;\; n_{1}+n_{2}+n_{3} \geq 2n_{1}+n_{2} \;\; \Rightarrow \;\; n_{3} \geq n_{1}.\]
Also,
\[  \beta \rho \geq an_{2}+cn_{3} \;\; \Rightarrow \;\;  \beta (n_{1}+n_{2}+n_{3}) \geq an_{2}+cn_{3} \] This means that
\[ 0 \geq -\beta n_{1} + (a - \beta) n_{2} + (c- \beta)n_{3} \geq (c- 2 \beta)n_{1} + (a - \beta)n_{2}. \] By assumption, ${\ds a- \beta >0}$. Also, Since  ${\ds a \leq  \left\lceil \frac{c}{2} \right\rceil}$, we get 
\[ \beta \leq  a -1 < \frac{c}{2}  \;\; \Rightarrow \;\;  c-2 \beta  > 0. \] Therefore $n_{1}=0=n_{2}$. This means that $\rho=n_{3}$ and $\beta \rho \geq c n_{3}$. Then $\beta \geq c$. Therefore,
\[ f = XZ^{\beta} = Z^{c}\cdot XZ^{\beta - c} \in J. \qedhere\]
\end{proof}

\medskip

In the following theorem, we consider a very specific monomial ideal in a polynomial ring with three variables. Although this case may appear highly specialized, it serves as a key building block for the more general results established later. Analyzing this ideal in detail allows us to identify structural properties and techniques that will be instrumental in extending our arguments to broader classes of homogeneous ideals.

\begin{thm}\label{dim3mono}
Let ${\ds R=k[X, Y, Z]}$ be a polynomial ring over a field $k$. Let 
\[ I= (X^{2},\, XY,\, Y^{2},\, Z^{c},\, XZ^{a},\, YZ^{b}),\]
where $a, b, c$ are integers such that ${\ds 1 \leq a \leq b}$ and  $c \geq 2$. Suppose that $I$ is integrally closed. Then ${\ds b \leq  \left\lceil \frac{c}{2} \right\rceil}$ and $I$ is normal.
\end{thm}

\begin{proof} Let ${\ds d=\left\lceil \frac{c}{2} \right\rceil}$ and suppose that ${\ds b > d  }$. Let ${\ds f = YZ^{d}}$. Then $f$ cannot be a multiple of $YZ^{b}$. 
Note that 
\[ d < \frac{c}{2} + 1 \leq c,\]
where the last inequality follows from $c \geq 2$. Then $f$ cannot be a multiple of $Z^{c}$ either. Therefore, ${\ds f \notin I}$. However, since ${\ds c \leq 2d}$, we get
\[ Z^{2d} = Z^{c} \cdot Z^{2d-c} \in I.\]
Then
\[ \left(YZ^{d}\right)^{2} - Y^{2} \cdot Z^{2d} =0, \]
where ${\ds Y^{2} \cdot Z^{2d} \in I^{2} }$. Thus, ${\ds f \in \overline{I} }$. This is a contradiction.

\medskip

\noindent To prove $I$ is normal, by \cite[Proposition 3.1]{RRV}, it is enough to prove that $I^{2}$ is integrally closed. Let ${\ds f = X^{\alpha}Y^{\beta}Z^{\gamma} }$ be a monomial in the integral closure ${\ds \overline{I^{2}}}$, where $\alpha, \beta, \gamma \geq 0$. We show that $f \in I^{2}$ by cases.

\medskip

\noindent {\it Case 1.} Suppose that $\alpha=0$ or $\beta =0$. First, we consider the case where $\beta =0$.  Note that, since ${\ds a \leq b}$ and ${\ds b \leq  \left\lceil \frac{c}{2} \right\rceil}$ , we have  ${\ds a \leq  \left\lceil \frac{c}{2} \right\rceil}$. By Lemma~\ref{dim2mono}, we obtain
\[ f = X^{\alpha}Z^{\gamma} \in \overline{ \left(X^{2}, XZ^{a}, Z^{c}  \right)^{2} } = \left(X^{2}, XZ^{a}, Z^{c}  \right)^{2} \subseteq I^{2}. \] Similarly, if $\alpha =0$, then ${\ds f \in I^{2}}$.

\medskip

\noindent {\it Case 2.} Suppose that $\alpha \geq 2$ and $\beta \geq 2$. Then
\[ f = X^{\alpha}Y^{\beta}Z^{\gamma} = X^{2} \cdot Y^{2} \cdot X^{\alpha-2}Y^{\beta -2}Z^{\gamma} \in I^{2}. \]

\medskip

\noindent {\it Case 3.} Suppose that $\alpha =1$ and $\beta \geq 3$. Then
\[ f = XY^{\beta}Z^{\gamma}= XY \cdot Y^{2} \cdot Y^{\beta -3}Z^{\gamma} \in I^{2}. \] Similarly, if $\beta=1$ and $\alpha \geq 3$, then ${\ds f \in I^{2}}$.

\medskip

\noindent {\it Case 4.} Suppose that ${\ds f= XYZ^{\gamma}}$ or ${\ds f= XY^{2}Z^{\gamma}}$ or ${\ds f= X^{2}YZ^{\gamma}}$. Let ${\ds \varphi: R=k[X, Y, Z] \rightarrow S=k[t]}$ be an arbitrary $k$-algebra map, where $k[t]$ is a polynomial ring. Since $S$ is a PID, we obatin,
\[ \varphi(f) \in \varphi (\overline{I^{2}}) \subseteq \overline{ \varphi(I^{2})S} = \varphi(I^{2})S.  \]
Note that ${\ds 2a-1 \leq 2b-1 \leq c}$. 

\medskip

{\it Case 4-1.} Suppose that $a=b$ and $c \geq 2a$. Let  ${\ds \varphi: R \rightarrow S=k[t]}$ be a $k$-algebra map given by ${\ds \varphi(X)=t^{a}}$, ${\ds \varphi(Y) = t^{a}}$, and ${\ds \varphi(Z)=t}$. Then since $c \geq 2a$, we get
\[ \varphi(I)S = (t^{2a},\, t^{c}) = (t^{2a}) \quad \mbox{and} \quad \varphi(I^{2})S = (t^{4a}).  \]
If ${\ds f= XYZ^{\gamma}}$, then ${\ds \varphi(f) = t^{2a+\gamma} \in (t^{4a})}$. Thus, ${\ds 2a + \gamma \geq 4a}$, which means ${\ds \gamma  \geq 2a}$. Therefore,
\[ f = XYZ^{\gamma} = XZ^{a} \cdot YZ^{a} \cdot Z^{\gamma - 2a} =  XZ^{a} \cdot YZ^{b} \cdot Z^{\gamma - 2a} \in I^{2}. \]
If ${\ds f= XY^{2}Z^{\gamma}}$ or ${\ds f= X^{2}YZ^{\gamma}}$, then ${\ds \varphi(f) = t^{3a+\gamma} \in (t^{4a})}$, which means ${\ds \gamma  \geq a}$.
Therefore,
\[ f = XY^{2}Z^{\gamma} = XZ^{a} \cdot Y^{2} \cdot Z^{\gamma - a} \in I^{2} \quad \mbox{or} \quad  f = X^{2}YZ^{\gamma} = XZ^{a} \cdot XY \cdot Z^{\gamma - a}    \in I^{2}. \] 

\medskip

{\it Case 4-2.} Suppose that $a=b$ and $c = 2a-1$. Let  ${\ds \varphi: R \rightarrow S=k[t]}$ be a $k$-algebra map given by ${\ds \varphi(X)=t^{c}}$, ${\ds \varphi(Y) = t^{c}}$, and ${\ds \varphi(Z)=t^{2}}$. Then since $2c < c+2a$, we get  
\[ \varphi(I)S = (t^{2c},\, t^{c+2a}) = (t^{2c}) \quad \mbox{and} \quad \varphi(I^{2})S = (t^{4c}).  \]
If ${\ds f= XYZ^{\gamma}}$, then ${\ds \varphi(f) = t^{2c+2\gamma} \in (t^{4c})}$. Thus, ${\ds 2c + 2\gamma \geq 4c}$, which means ${\ds \gamma  \geq c}$. Therefore,
\[ f = XYZ^{\gamma} = XY \cdot Z^{c} \cdot Z^{\gamma - c}  \in I^{2}. \]
If ${\ds f= XY^{2}Z^{\gamma}}$ or ${\ds f= X^{2}YZ^{\gamma}}$, then ${\ds \varphi(f) = t^{3c+2\gamma} \in (t^{4c})}$. Thus, ${\ds 3c + 2\gamma \geq 4c}$, which means ${\ds \gamma  \geq a}$. Therefore,
\[ f = XY^{2}Z^{\gamma} = XZ^{a} \cdot Y^{2} \cdot Z^{\gamma - a} \in I^{2} \quad \mbox{or} \quad  f = X^{2}YZ^{\gamma} = XZ^{a} \cdot XY \cdot Z^{\gamma - a}  \in I^{2}. \]

\medskip

{\it Case 4-3.} Suppose that $a<b$. If ${\ds f= XYZ^{\gamma}}$, then we consider the $k$-algebra map ${\ds \varphi: R \rightarrow S=k[t]}$ given by ${\ds \varphi(X)=t^{q}}$, ${\ds \varphi(Y) = t^{c}}$, and ${\ds \varphi(Z)=t^{2}}$, where ${\ds q= 2(b-a)+c-1 }$. Then 
\[ \varphi(I)S= ( t^{2q},\, t^{q+c},\, t^{2c},\, t^{q+2a},\, t^{c+2b}).\]
Note that ${\ds q \geq c+1}$. Then
\[ 2q \geq 2c+2 > 2c \quad \mbox{and} \quad q+c \geq 2c+1 > 2c.\]
On the other hand,
\[ q+2a = 2b-1+c \leq 2c.\]
Moreover, we have 
\[ c+2b - (q+2a) = c+2b - 2(b-a)-c+1 - 2a = 1, \] which shows that 
${\ds q+2a <  c+2b}$. Therefore,
\[ \varphi(I)S= (t^{q+2a}) \quad \mbox{and} \quad \varphi(I^{2})S= (t^{2q+4a}) \]
Then ${\ds \varphi(f) = t^{q+c+2\gamma} \in (t^{2q+4a})  }$. Thus, ${\ds q+c+2 \gamma \geq 2q+4a}$, which means ${\ds \gamma \geq a+b }$. Therefore,
\[ f= XYZ^{\gamma} = XZ^{a} \cdot YZ^{b} \cdot Z^{\gamma - a- b} \in I^{2}. \]
Finally, we consider 
the $k$-algebra map ${\ds \varphi: R \rightarrow S=k[t]}$ given by ${\ds \varphi(X)=t^{a}}$, ${\ds \varphi(Y) = t^{a}}$, and ${\ds \varphi(Z)=t}$. Then 
\[ \varphi(I)S= ( t^{2a},\, t^{c},\, t^{a+b}).\]
Note that ${\ds 2a < a+b \leq 2b-1 \leq c}$. Thus,
\[ \varphi(I)S= (t^{2a}), \quad \mbox{and} \quad \varphi(I^{2})S= (t^{4a}).\]
If ${\ds f = XY^{2}Z^{\gamma}}$ or $f = X^{2}YZ^{\gamma}$, then ${\ds \varphi(f) = t^{3a+\gamma} \in (t^{4a})}$. Thus,  ${\ds \gamma  \geq a}$. Hence,
\[ f=XY^{2}Z^{\gamma} = XZ^{a} \cdot Y^{2} \cdot Z^{\gamma - a} \in I^{2}  \quad \mbox{and} \quad f=X^{2}YZ^{\gamma} = XZ^{a} \cdot XY \cdot Z^{\gamma -a} \in I^{2}. \qedhere \]
\end{proof}

\medskip

Before addressing  the normality of more general homogeneous ideals, we first establish the following lemma. It gives a precise characterization of homogenous ideals that do not contain a quadratic form of rank three. This characterization will serve as a key tool in our subsequent analysis, providing a foundation for the results on normality that follow.

\begin{lem}\label{grade2}
Let $k$ be an algebraically closed field, $R=k[X, Y, Z]$ a polynomial ring over $k$, and $\m$ the homogeneous maximal ideal of $R$. Let $I$ be an integrally closed $\m$-primary ideal such that $I \subseteq \m^{2}$ and $\mu_{R}(I)=6$. Then ${\ds \grade(I_{2}R) \geq 2}$, where $I_{2}$ is the set of all quadratic forms in $I$.
\end{lem}

\begin{proof}
Suppose ${\ds \grade(I_{2}R)=0}$. Then, since $R$ is an integral domain,  $I_{2}R=(0)$. Therefore, ${\ds I \subseteq \m^{3}}$. Because $I_{\m}$ is $\m R_{\m}$--full, we obtain the following:
\[ 6 \geq \mu_{R_{\m}}(I_{\m}) \geq \mu_{R_{\m}}(\m^{3}R_{\m}) = 10, \]
which is a contradiction.

\medskip

\noindent Suppose ${\ds \grade(I_{2}R)=1}$. Then there exists ${\ds \p \in \Spec(R)}$ such that $\height(\p)=1$ and $I_{2}R \subseteq \p$. Since $R$ is a UFD, $\p$ is a principal ideal, say $\p=(g)$ for some nonzero nonunit $g \in R$. Since $k$ is algebraically closed, ${\ds g= g_{1} g_{2} \cdots g_{l}}$, where each $g_{i}$ is an irreducible polynomial of degree $1$ in $R$. Since $\p$ is a prime ideal, $g_{t} \in \p$ for some $t=1, \ldots, l$. Then ${\ds (0) \subsetneq (g_{t}) \subseteq \p}$ and ${\ds (g_{t}) \in \Spec(R)}$. Since $\height(\p)=1$, we get ${\ds \p=(g_{t})}$.  We denote the set of all quadratic forms in $R$ by $R_{2}$. Then
\[ I \subseteq I_{2}R \cap R_{2} + \m^{3} \subseteq \p \cap R_{2} + \m^{3} \subseteq g_{t} \m + \m^{3}. \]
Since $I_{\m}$ is $\m R_{\m}$-full, we obtain the following:
\[ 6 \geq \mu_{R_{\m}} (I_{\m}) \geq \mu_{R_{\m}} \left(  \left( g_{t} \m + \m^{3} \right)R_{\m} \right) =7, \]
which is a contradiction.
\end{proof}

\medskip

\begin{thm}\label{Ulrich}
Let $k$ be a field of characteristic zero, $R=k[X_{1}, \ldots, X_{d}]$ a polynomial ring over $k$, and $I$ a zero-dimensional $R$-ideal generated by $d+3$ homogeneous polynomials. If $I$ is integrally closed, then the Rees algebra $\calR(I)$ is a Cohen-Macaulay normal domain. 
\end{thm}

\begin{proof}
We may assume that $k$ is algebraically closed. Moreover, it is enough to show that $I$ is normal due to \cite[Theorem 1]{Hoc72}. We can further suppose that $d \geq 3$ and $\mu_{R}(I)=d+3$. Write $\fkm$ for the homogeneous maximal ideal of $R$. If $d \geq 4$, then using the $\m R_{\m}$-fullness of $I_{\m}$ and the same techniques given in the proof of Theorem~\ref{div2},  we may assume that $X_{4}, \ldots, X_{d}$ is a part of regular system of parameters of length $d-3$ in $I$. Then we can write ${\ds I= (J, X_{4}, \ldots, X_{d})}$ for some zero-dimensional ideal $J$ of $k[X_{1}, X_{2}, X_{3}]$. Since $I$ is integrally closed (or normal) if and only if $J$ is integrally closed (or normal respectively), we may assume that $I=J$ and $d=3$. 

\medskip

\noindent Let ${\ds R=k[X, Y, Z]}$, ${\ds \m = (X, Y, Z)}$, and let ${\ds f_{1}, \ldots, f_{6}}$ be a homogeneous minimal generating set of $I$. If $I \not \subset \m^{2}$, then $I$ contains a part of regular system of parameters of length $1$. By Theorem~\ref{main}, $I$ is normal. Therefore, we suppose that ${\ds I \subseteq \fkm^{2}}$. 

\medskip

\noindent We first treat the case where $I_{2}$ contains a quadratic form of rank $3$. Consider the purely transcendental field extension ${\ds k \subset k''=k(T_{1}, \ldots, T_{6})}$, and let ${f= \sum_{i=1}^{6} T_{i}f_{i} }$ be a generic element for $I$ in the ring ${\ds R''= k''[X, Y, Z]_{(X, Y, Z)}}$. Write $\fkm''$ for the maximal ideal of $R''$. Then 
\[ \gr_{\fkm''/(f)} (R''/(f)) = k''[X, Y, Z]/(f^{*}), \]
where $f^{*}$ is the homogeneous component of degree $2$ of $f$. Notice that $f^{*}$ is a generic element for $I_{2}$.  Since $I_{2}$ contains a quadratic form of rank $3$, we then conclude that $f^{*}$ has rank $3$ as well. In other words, ${\ds \gr_{\fkm''/(f)} (R''/(f))}$ is an isolated singularity. Since the $a$-invariant of this standard graded $k''$-algebra is $-1$,  ${\ds \gr_{\fkm''/(f)} (R''/(f))}$ is a  rational singularity (See \cite[Satz 3.1]{F81}). Then $R''/(f)$ is a $2$-dimensional rational singularity according to \cite[Satz 3.5]{F81}. On the other hand, \cite[Theorem 2.1]{HU} shows that $IR''/(f)$ is still integrally closed. By \cite[Theorem 7.1]{L69}, $IR''/(f)$ is normal. Now the proof of Theorem~\ref{main} shows that $I$ is normal.

\medskip

\noindent We now assume that $I_{2}$ does not contain a quadratic form of rank $3$. We claim that $I_{2}$ contains a quadratic form of rank $2$. To prove the claim, suppose that every nonzero element of $I_{2}$ has rank $1$. Then every such element is the square of a linear form.  Since ${\ds \grade I_{2}R \geq 2}$ by Lemma~\ref{grade2}, two of those linear forms have to be linearly independent and then the sum of their squares has rank $2$. This is a contradiction. Thus, $I_{2}$ contains a quadratic form $g$ of rank $2$. After a linear change of variables we may assume that $g=2XY$. Since ${\ds \grade I_{2}R \geq 2}$, there exists an element $h$ of $I_{2}$ such that $g$ and $h$ form a regular sequence. The quadratic form $h$ is represented by a matrix 
\[ M = \left[ \begin{array}{ccc} \alpha & \beta & \gamma \\ \beta & \delta & \epsilon \\ \gamma & \epsilon & \nu \end{array} \right]\]
whose entries are in $k$. By our assumption, ${\ds \rank(\lambda g + h) \leq 2}$  for every $\lambda \in k$. Equivalently, 
\[ \det \left[ \begin{array}{ccc}  \alpha & T+\beta & \gamma \\ T+\beta & \delta & \epsilon \\ \gamma & \epsilon & \mu \end{array} \right] =0 \]
for a variable $T$. Considering the coefficients of $T^{2}$ and $T$ in this polynomial, we obtain $\nu =0$ and $\epsilon \gamma  =0$. Without loss of generality we may assume that $\gamma =0$. Then
\[ 0 = \det(M) = - \alpha \epsilon^{2}. \]
If $\alpha=0$, then ${\ds h = 2 \beta XY + \delta Y^{2} + 2 \epsilon YZ}$, contradicting the assumption that $g$ and $h$ form a regular sequence. It follows that $\epsilon=0$. In other words, $h \in k[X, Y]$. 

\medskip

\noindent  Thus, $g$ and $h$ form a regular sequence of quadrics in $k[X, Y]$. Therefore,
\[ (X, Y)^{2} k [X, Y] = \overline{(g, h)k[X, Y]} \subseteq \overline{I} = I .\]
Then we may assume that ${\ds f_{1}=X^{2}}$, ${\ds f_{2}=XY}$, and ${\ds f_{3}=Y^{2}}$, and that the remaining generators of $I$ are of the form 
\[ f_{4} = Z^{c} + l_{4} Z^{c-1}, \quad f_{5} = l_{5} Z^{a}, \quad f_{6} = l_{6} Z^{b},\]
where $l_{4}$, $l_{5}$, $l_{6}$ are linear forms in $k[X, Y]$, $c \geq 1$ and $a \leq b$. As ${\ds \mu_{R}(I) = 6}$, it follows that $l_{5}$ and $l_{6}$ are linearly independent. Thus, after a linear change of variables in $k[X, Y]$, we obtain 
\[ f_{5} = XZ^{a}, \quad \mbox{and} \quad f_{6} = YZ^{b}. \]
If $a \geq c$, then 
\[ f_{5}  = XZ^{a-c}f_{4} - X l_{4} Z^{a-1} \in (X^{2}, XY,  Y^{2}, f_{4}),\]
contradicting the assumption ${\ds \mu_{R}(I) = 6}$. Thus, $a <c$, and likewise $b <c$. Now, adding suitable multiples of $f_{5}$ and $f_{6}$ to $f_{4}$, we achieve that $f_{4} = Z^{c}$. In conclusion, we have proved that 
\[ I = (X^{2}, XY, Y^{2}, Z^{c}, XZ^{a},  YZ^{b}), \]
where $1 \leq a \leq b \leq c-1$. By Theorem~\ref{dim3mono}, $I$ is normal.
\end{proof}

\begin{ex}
Let $R=k[\![X, Y, Z]\!]$ be the formal power series ring over a field $k$ of characteristic zero. We consider $I=\overline{(X^2, Y^2, Z^4)} = (X^2, XY, Y^2, Z^4, XZ^2, YZ^2) \subseteq \m^2$. Then $I$ is an integrally closed $\m$-primary ideal of $R$ and $\mu_R(I) = 6$, where $\m = (X, Y, Z)$.  Theorem~\ref{Ulrich} shows that the Rees algebra $\calR(I)$ is a Cohen-Macaulay normal domain.
\end{ex}

\section{Rees Algebra of Modules}

Let $(R, \m)$ be a Noetherian local ring. Let $E$ be a finitely generated torsionfree $R$-module having  a rank $e >0$. Suppose that $E_{\fkp}$ is free for all $\fkp \in \Spec(R)$ with $\depth(R_{\fkp}) \leq 1$. Then by \cite[Lemma 4.1]{HU}, there exists an embedding $E \subset R^{e}$ such that $(R^{e}/E)_{\fkp}$ is cyclic whenever $\depth(R_{\fkp}) \leq 1$.  

\medskip

Write ${\ds E=Ra_{1} + \cdots + Ra_{n}}$ and let ${\ds R'=R[ \{ z_{ij} \mid 1 \leq i \leq n, \; 1 \leq j \leq e-1 \}]}$ be a polynomial ring with indeterminates $z_{ij}$. Set
\[ R''= R'_{\m R'}, \quad E'' = R'' \otimes_{R} E, \quad  x_{j}= \sum_{i=1}^{n} z_{ij} a_{i} \in E'', \quad F = \sum_{j=1}^{e-1} R''x_{j}. \]
Then $F$ is a free $R''$-module of rank $e-1$ and $E''/F \simeq I$ for some $R''$-ideal $I$ with $\grade(I) >0$. This ideal $I=I(E)$ is called a {\em generic Bourbaki ideal} of $E$ (\cite[Proposition 3.2, Definition 3.3]{SUV}).

\medskip

 Suppose that $\ell_R(R^{e}/E)< \infty$. Then the {\em deviation} of $E$ is  ${\ds d(E) = \mu_{R}(E) -e+1 - d}$. For more general definition of the deviation, see \cite[Page 633]{SUV}. In particular, a module $E$ is called a {\it complete intersection} (respectively {\it almost complete intersection}) if $d(E)=0$ (respectively $d(E)$=1). Consider the commutative diagram 
\[ 
\xymatrix{
0 \ar[r]&  F \ar[r]\ar[d]^{\simeq} & E'' \ar[r]\ar[d]^{i} &E''/F \cong I \ar[d]\ar[r] & 0\\
0\ar[r] & (R'')^{e-1} \ar[r] & (R'')^{e} \ar[r] &  R''\ar[r]& 0
}
\] Then  ${\ds \ell_{R''}(R''/I) = \ell_{R''} \left( (R'')^{e}/E'' \right) = \ell_{R} (R^{e}/E) < \infty}$. 
Thus, $I$ is $\m''$-primary.

\medskip

\begin{remark}\label{BI}{\rm
Let $(R, \m)$ be a Cohen-Macaualy local ring of dimension $d \ge 3$. Let $E$ be a finitely generated torsionfree $R$-module having  a rank $e >0$. Suppose that $E_{\fkp}$ is free for all $\fkp \in \Spec(R)$ with $\depth(R_{\fkp}) \leq 1$ and that $\ell_R(R^{e}/E)< \infty$. Let $I \simeq E''/F$ be the generic Bourbaki ideal of $E$. Then $\grade(I) \geq 3$, and $E'' \simeq F \oplus I$ (See the proof of \cite[Remark 3.4-(d)]{SUV}.) 
}\end{remark}

Let $(R, \m)$ be a regular local ring of dimension $2$. Let $E$ be a finitely generated torsionfree $R$-module. If $E$ is integrally closed, then the Rees algebra $\calR(E)$ is a Cohen-Macaulay normal domain (See \cite[Corollary 3.8]{HU} and \cite[Theorem 5.2]{K}). Now we consider a regular local ring of dimension greater than $2$.

\begin{thm}\label{mod}
Let $(R, \m)$ be a regular local ring of dimension $d \geq 3$. Let $E$ be a finitely generated torsionfree $R$-module having a rank $e >0$.  Suppose that $E$ is integrally closed and  $\ell_R(R^{e}/E)< \infty$.
If $\mu_R([E + \m R^{e}]/E) \leq 2$ or $d(E) \leq 2$, then  the Rees algebra  $\calR(E)$ is a Cohen-Macaulay normal domain.
\end{thm}

\begin{proof} 
Let $I \simeq E''/F$ be the generic Bourbaki ideal of $E$. By \cite[Theorem 4.4]{HU}, $I$ is integrally closed.  By Remark~\ref{BI}, $E'' \simeq F \oplus I$. Suppose that $\mu_R([E + \m R^{e}]/E) \le 2$. Consider the isomorphism:
\[ [E'' + \m (R'')^{e}]/E'' \simeq (F'' \oplus \m)/(F'' \oplus I) \simeq \m''/I. \]
Then $\mu_{R''}(\m''/I) \le 2$. Equivalently, $I$ contains a part of regular system of parameters of length $d-2$. Then, by Theorem~\ref{main}, the Rees algebra $\calR(I)$ is a Cohen-Macaulay normal domain. By \cite[Theorem 3.5]{SUV}, the Rees algebra $\calR(E)$  is a Cohen-Macaulay normal domain. 

\medskip

\noindent Suppose that $d(E) \leq 2$. Then 
\[ d(E) = \mu_{R}(E) -e+1 - d  = e-1 + \mu_{R''}(I)  - e+1-d =  \mu_{R''}(I)  - d \leq 2. \]
By Theorem~\ref{div2}, the Rees algebra $\calR(I)$ is a Cohen-Macaulay normal domain. By \cite[Theorem 3.5]{SUV}, the Rees algebra $\calR(E)$  is a Cohen-Macaulay normal domain. 
\end{proof}

\begin{ex}
Let $R=k[\![X, Y, Z]\!]$ be the formal power series ring over a field $k$. Let $I=(f) + \m^n$ for each $f \in \m\setminus \m^2$ and $n \ge 1$. For each $e > 0$, let $E=I \oplus \m^{\oplus (e-1)}$. 
Then $E$ is integrally closed, $\ell_R(R^{e}/E) <\infty$, and $\mu_R([E+ \m R^{e}]/E) \le 2$. Hence, $\calR(E)$ is a Cohen-Macaulay normal domain. 
\end{ex}

Recall that a $R$-module $E$ is called a {\it parameter module} if there is an embedding $E \subseteq R^{e}$ with  $\ell_R(R^{e}/E) < \infty$, and $\mu_{R}(E)=d+e-1$ (or $d(E)=0$).

\begin{cor}
Let $(R, \m)$ be a regular local ring of dimension $d \geq 2$ and $E$ a parameter module. If $E$ is integrally closed, then $\calR(E)$ is a Cohen-Macaulay normal domain.
\end{cor}




\end{document}